\documentclass[3p,12pt]{elsarticle} 
\usepackage[utf8]{inputenc}
\usepackage[T1]{fontenc}
\usepackage{lmodern}

\usepackage{microtype}

\usepackage{amsmath,amssymb,amsthm}

\usepackage{graphicx}
\usepackage{caption}
\usepackage{subcaption}
\usepackage{float}
\usepackage{booktabs}
\usepackage{multirow}
\usepackage{enumitem}

\usepackage[hidelinks]{hyperref}
\hypersetup{pdftitle={Monotonicity of the period function},pdfauthor={F. J. S. Nascimento}}

\usepackage{siunitx}

\theoremstyle{plain}
\newtheorem{theorem}{Theorem}[section]
\newtheorem{lemma}[theorem]{Lemma}
\newtheorem{proposition}[theorem]{Proposition}

\theoremstyle{definition}

\theoremstyle{remark}
\newtheorem{remark}[theorem]{Remark}
\newtheorem{example}[theorem]{Example}

\makeatletter
\def\ps@pprintTitle{%
 \let\@oddhead\@empty
 \let\@evenhead\@empty
 \def\@oddfoot{}%
 \def\@evenfoot{}}
\makeatother
\begin{document}

\begin{frontmatter}

\title{Monotonicity of the period function for planar Hamiltonian vector fields:\\
\large A Generalization of Chicone's Criterion}

\author[1]{F.J. S. Nascimento\corref{cor1}}
\ead{francisco.jsn@univasf.edu.br}
\cortext[cor1]{Corresponding author}
\address[1]{Universidade Federal do Vale do S\~ao Francisco, Colegiado de Geologia, Campus Senhor do Bonfim, BA, Brazil}

\begin{abstract}
This paper investigates the monotonicity of the period function associated with planar Hamiltonian systems of the form \(H(x,y) = F(x) + G(y)\). 
We establish sufficient conditions ensuring the monotonicity of the period function corresponding to a nondegenerate center, expressed explicitly in terms of the functions \(F\) and \(G\).
Our approach extends Chicone’s classical criterion, originally formulated for systems of the type \(x' = y,\ y' = -F'(x)\), to a broader Hamiltonian framework.
As a main application, we analyze the monotonicity of the period function associated with the center at \((0,0)\) of the polynomial Hamiltonian system
\[
H(x,y) = \frac{1}{2}x^2 + \frac{a}{3}x^3 + \frac{b}{4}x^4 + \frac{1}{2}y^2 + \frac{c}{4}y^4,
\]
as a function of the parameters \(a, b, c \in \mathbb{R}\).
\end{abstract}


\end{frontmatter}



\section{Introduction and Main Results}
We consider classical planar Hamiltonian systems with a Hamiltonian function of the form
\begin{equation}
H(x, y) = F(x) + G(y),
\end{equation}
where \( F(x) \) and \( G(y) \) are smooth functions, and \( G \) satisfies \( G(-y) = G(y) \), that is, \( G \) is even. 
Assume that \( F(0) = G(0) = F'(0) = G'(0) = 0 \), \( F''(0) > 0 \), and \( G''(0) > 0 \).
Under these conditions, the Hamiltonian system
\begin{equation}\label{eq2}
\begin{cases}
x' = G'(y), \\ 
y' = -F'(x),
\end{cases}
\end{equation}
has a nondegenerate center at \( (0, 0) \), and its trajectories lie on the level curves \( H(x, y) = E \).

Let \( \mathcal{P} \) denote the period annulus of the center, that is, the largest neighborhood of \( (0, 0) \) entirely filled with periodic orbits.
It can be shown (see~\cite{cima2000period}, for instance) that \( H(\mathcal{P}) = [0, E^*) \), where \( E^* \in (0, +\infty] \).
If \( E^* \) is finite, then the boundary of the period annulus is contained in the energy level \( H(x, y) = E^* \).
Moreover, the set of all periodic orbits in the period annulus can be parametrized by the energy, which allows us to define the period function \( T: (0, E^*) \to (0, +\infty) \), assigning to each periodic orbit \( \gamma_E \subset \mathcal{P} \) its minimal period \( T(E) \).
When the origin is a nondegenerate center, it can be shown that \( T(E) \) is a smooth function satisfying \( T(0) > 0 \).

Let \( 0 < E_0 < E^* \).  
We say that \( T \) is \emph{monotonically increasing} on \( (0, E_0) \) if \( T'(E) \geq 0 \) for all \( E \in (0, E_0) \), and \emph{monotonically decreasing} on \( (0, E_0) \) if \( T'(E) \leq 0 \) for all \( E \in (0, E_0) \).
If \( T'(E) = 0 \) on \( (0, E^*) \), then \( T \) is constant on \( (0, E^*) \), and in this case, \( \mathcal{P} \) is said to be an \emph{isochronous center}.
In this article, we investigate the monotonicity of the period function \( T(E) \) associated with the Hamiltonian system~\eqref{eq2}.
Before stating the main results, we introduce some notation.

Let \( \gamma_E \) be a periodic orbit contained in \( \mathcal{P} \), corresponding to the level set \( H = E \).
This orbit intersects the axis \( y = 0 \) (resp. \( x = 0 \)) at the points determined by \( F(x_E(t)) = E \) (resp. \( G(y_E(t)) = E \)).
Since \( F \) has a minimum at \( x = 0 \), near the origin the equation \( F(x) = E \) has two solutions: one for \( x > 0 \), denoted by \( x_+ = x_+(E) \), and one for \( x < 0 \), denoted by \( x_- = x_-(E) \) (similarly, \( y_+ = y_+(E) \) for \( y > 0 \) and \( y_- = y_-(E) \) for \( y < 0 \)) (see Fig.~\ref{Fig1}).

\begin{figure}[hbt!]
    \centering
    \includegraphics[width=0.5\linewidth]{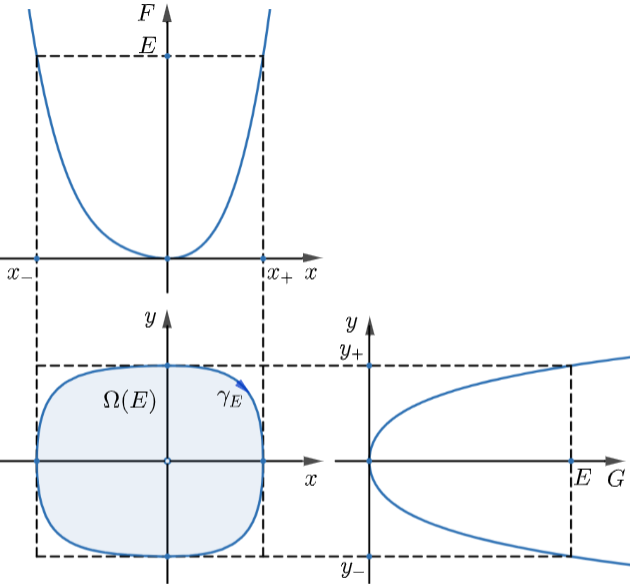}
    \caption{Periodic orbit \( \gamma_E \) of energy \( E \) and the region \( \Omega(E) \).}
    \label{Fig1}
\end{figure}

To state our main result, define \( \Omega(E) \) and \( M \) by
\[
\Omega(E) = \{(x, y) \in \mathbb{R}^2 : 0 \leq H(x, y) \leq E\},
\]
and
\begin{equation*}\label{eq:N}
\begin{split}
M(x, y) = {} & \big[6F(x)(F''(x))^2 - 3(F'(x))^2F''(x) - 2F(x)F'(x)F'''(x)\big] G(y)\,(G'(y))^2 \\
& + F(x)\,(F'(x))^2F''(x)\,\big[(G'(y))^2 - 2G(y)G''(y)\big].
\end{split}
\end{equation*}

Our main result is stated in the following theorem.

\begin{theorem}\label{th1}
Let \( 0 < E_0 \leq E^* \).  
If \( M(x, y) \geq 0 \) for all \( (x, y) \in \Omega(E_0) \), then \( T \) is monotonically increasing on \( (0, E_0) \).  
If \( M(x, y) \leq 0 \) for all \( (x, y) \in \Omega(E_0) \), then \( T \) is monotonically decreasing on \( (0, E_0) \).
\end{theorem}

The monotonicity of the period function has been extensively studied in the literature (see, for instance, \cite{Hsu1983, rothe1985periods, schaaf1985class, chow1986monotonicity, chicone1987monotonicity, chicone1988quadratic, chicone1988geometric, rothe1993remarks, chicone1993finiteness, coppel1993period, gasull1997period, cima1999isochronicity, freire2004first, sabatini2004period, gasull2004period, schaaf2006global}). 
This subject plays a fundamental role in the study of bifurcation phenomena \cite{chow1986number, chicone1989bifurcation, chicone1992bifurcations, gasull2008critical} and in questions related to inverse problems \cite{urabe1961potential, urabe1964relations, schaaf2006global, rocha2007realization, ragazzo2012scalar, grotta2025global}.

In the potential case, that is, when \( G(y) = y^2/2 \), the monotonicity of the period function was analyzed in \cite{chow1986monotonicity, chicone1987monotonicity}. 
The particular case \( F' = G' \), under additional assumptions, was studied in \cite{schaaf1985class}, while more general approaches can be found in \cite{freire2004first, Sabatini2006, villadelprat2020period}.

In this potential setting, the main criterion for studying the monotonicity of the period function was established by Chicone (see \cite{chicone1987monotonicity}). 
Chicone proved that the period function \( T(E) \) is monotonically increasing on \( (0, E_0) \) if \( N(x) \geq 0 \) for all \( x \in K(E_0) \), and monotonically decreasing on \( (0, E_0) \) if \( N(x) \leq 0 \) for all \( x \in K(E_0) \), where
\[
K(E_0) = \{x \in \mathbb{R} : F(x) \leq E_0\} = [x_-(E_0), x_+(E_0)],
\]
and
\[
N(x) = 6F(x)(F''(x))^2 - 3(F'(x))^2F''(x) - 2F(x)F'(x)F'''(x).
\]

Theorem~\ref{th1} provides a natural generalization of Chicone’s criterion to a broader class of Hamiltonian systems of the form \( H(x, y) = F(x) + G(y) \). 
Indeed, if \( G(y) = \tfrac{1}{2}y^2 \), then \( (G')^2 - 2GG'' = 0 \), and consequently \( M(x, y) = N(x)y^4/2 \). 
Since \( y^4/2 > 0 \) for \( y \neq 0 \), the sign of \( M(x, y) \) coincides with the sign of \( N(x) \). 
Thus, Theorem~\ref{th1} reduces to Chicone’s original criterion in the potential case, while providing a unified framework for the more general separable Hamiltonian setting.

\begin{proof}
We begin by outlining the proof strategy. 
To establish Theorem~\ref{th1}, it suffices to prove that 
\(\tfrac{dT}{dE} \geq 0\) (respectively, \(\leq 0\)) on \((0, E_0)\) 
whenever \(M(x, y) \geq 0\) (respectively, \(\leq 0\)) throughout \(\Omega(E_0)\). 
The proof proceeds as follows. First, we express the period function \(T(E)\) as an integral. 
Next, we differentiate under the integral sign to obtain an explicit expression for \(\tfrac{dT}{dE}\). 
Finally, we show that the sign of this derivative is determined by the function \(M(x, y)\).

Since \(G\) is an even function, the phase portrait of system~\eqref{eq2} is symmetric with respect to the \(x\)-axis. 
Hence, the period of the periodic orbit with energy \(E\) is given by the convergent improper Riemann integral
\[
T(E) = 2 \int_{x_0}^{x_1} \frac{dx}{G'(y)} 
     = 2 \int_{x_0}^{x_1} \frac{dx}{G'(G^{-1}(E - F(x)))},
\]
where \(x_0\) and \(x_1\) denote the left and right intersection points of the level curve \(H(x,y) = E\) with the \(x\)-axis.

The evenness of \(G\) implies the existence of a smooth function \(g\) satisfying \(g(0) = 0\) and \(g'(0) > 0\) such that (see~\cite{Whitney1943})
\begin{equation}\label{eq_1}
G(y) = g\!\left(\frac{y^2}{2}\right).
\end{equation}
Differentiating~\eqref{eq_1} yields
\[
G'(y) = y\, g'\!\left(\frac{y^2}{2}\right),
\]
and consequently,
\[
G^{-1}(z) = \pm \sqrt{2\, g^{-1}(z)},
\]
where \(g^{-1}\) denotes the inverse function of \(g\). 
Since \(g^{-1}(0) = 0\), the fundamental theorem of calculus implies
\[
g^{-1}(z) = z\, f^2(z),
\]
where
\begin{equation}\label{eq_2}
f^2(z) := \int_0^1 (g^{-1})'(t z)\, dt > 0.
\end{equation}

Substituting these expressions into the period integral gives
\begin{align*}
T(E) 
&= 2 \int_{x_0}^{x_1} 
   \frac{dx}{G^{-1}(E - F(x)) \, g'\!\big(g^{-1}(E - F(x))\big)} \\[4pt]
&= 2 \int_{x_0}^{x_1} 
   \frac{dx}{\sqrt{2\, g^{-1}(E - F(x))} \, g'\!\big(g^{-1}(E - F(x))\big)} \\[4pt]
&= 2 \int_{x_0}^{x_1} 
   \frac{dx}{\sqrt{2(E - F(x))} \, f(E - F(x)) \, g'\!\big(g^{-1}(E - F(x))\big)}.
\end{align*}

To facilitate the computation of \( \frac{dT}{dE} \), we introduce a change of variables. 
Define the auxiliary function \(h(x)\) by
\[
h(x) = 
\begin{cases} 
x \left(\frac{F(x)}{x^2}\right)^{1/2}, & \text{for } x \neq 0, \\[6pt]
0, & \text{for } x = 0.
\end{cases}
\]
Since \(F\) has a nondegenerate quadratic minimum at the origin, \(h(x)\) is smooth. 
Moreover, one can verify that \(h'(x) > 0\) for all \(x\), so that \(h\) is strictly increasing. 
Setting \(r = h(x)\) and performing the corresponding change of variables in the integral, we obtain
\[
T(E) = 2 \int_{-\sqrt{E}}^{\sqrt{E}} 
\frac{dr}{
\sqrt{2(E - r^2)}\, 
f(E - r^2)\, 
g'\big(g^{-1}(E - r^2)\big)\, 
h'\big(h^{-1}(r)\big)}.
\]

This expression naturally suggests the trigonometric substitution
\[
r = \sqrt{E} \sin\theta, \qquad -\frac{\pi}{2} \leq \theta \leq \frac{\pi}{2}.
\]
After this change of variables, the integral becomes
\begin{align*}
T(E) 
&= 2 \int_{-\pi/2}^{\pi/2} 
\frac{\sqrt{E} \cos\theta \, d\theta}{
\sqrt{2E \cos^2\theta}\, 
f(E \cos^2\theta)\, 
g'\big(g^{-1}(E \cos^2\theta)\big)\, 
h'\big(h^{-1}(\sqrt{E} \sin\theta)\big)} \\[4pt]
&= \frac{2}{\sqrt{2}} 
\int_{-\pi/2}^{\pi/2} 
\frac{d\theta}{
h'\big(h^{-1}(\sqrt{E} \sin\theta)\big)\,
g'\big(g^{-1}(E \cos^2\theta)\big)\,
f(E \cos^2\theta)}.
\end{align*}

With the integral in this simplified form, we can now differentiate \(T(E)\) with respect to \(E\):
\begin{align*}
\frac{dT}{dE} 
&= \frac{2}{\sqrt{2}} 
\int_{-\pi/2}^{\pi/2} 
\frac{d}{dE} 
\left( 
\frac{1}{
h'\big(h^{-1}(\sqrt{E} \sin\theta)\big)\,
g'\big(g^{-1}(E \cos^2\theta)\big)\,
f(E \cos^2\theta)} 
\right) d\theta \\[4pt]
&= -\frac{2}{\sqrt{2}} 
\int_{-\pi/2}^{\pi/2} 
\frac{
\left[
h'\big(h^{-1}(\sqrt{E} \sin\theta)\big)\,
g'\big(g^{-1}(E \cos^2\theta)\big)\,
f(E \cos^2\theta)
\right]'
}{
\left[
h'\big(h^{-1}(\sqrt{E} \sin\theta)\big)\,
g'\big(g^{-1}(E \cos^2\theta)\big)\,
f(E \cos^2\theta)
\right]^2
} \, d\theta.
\end{align*}

Define
\[
A(E) =
h'\!\left(h^{-1}(\sqrt{E} \sin\theta)\right)
\, g'\!\left(g^{-1}(E \cos^2\theta)\right)
\, f(E \cos^2\theta),
\]
so that
\[
\frac{dT}{dE}
= -\frac{2}{\sqrt{2}}
\int_{-\pi/2}^{\pi/2}
\frac{A'(E)}{A^2(E)} \, d\theta.
\]
Differentiating \(A(E)\) with respect to \(E\) yields
\begin{align*}
A'(E)
&=
\Bigg\{
\frac{
h''\!\left(h^{-1}(\sqrt{E} \sin\theta)\right)
\, g'\!\left(g^{-1}(E \cos^2\theta)\right)
\, f(E \cos^2\theta)
\, \sin\theta
}{
2\sqrt{E}\,
h'\!\left(h^{-1}(\sqrt{E} \sin\theta)\right)
}
\\[4pt]
&\quad +
h'\!\left(h^{-1}(\sqrt{E} \sin\theta)\right)
\bigg[
\frac{
g''\!\left(g^{-1}(E \cos^2\theta)\right)
\, f(E \cos^2\theta)
\, \cos^2\theta
}{
g'\!\left(g^{-1}(E \cos^2\theta)\right)
}\\[4pt]
&\quad+
g'\!\left(g^{-1}(E \cos^2\theta)\right)
\, f'(E \cos^2\theta)
\, \cos^2\theta
\bigg]
\Bigg\}.
\end{align*}

Consequently,
\[
\frac{dT}{dE}
= -\frac{2}{\sqrt{2}}
\left(
I_1(E) + I_2(E) + I_3(E)
\right),
\]
where the three integrals are given by
\begin{align*}
I_1(E)
&=
\int_{-\pi/2}^{\pi/2}
\frac{
h''\!\left(h^{-1}(\sqrt{E} \sin\theta)\right)
\, \sin\theta
}{
2\sqrt{E}\,
\big(h'\!\left(h^{-1}(\sqrt{E} \sin\theta)\right)\big)^3
\, g'\!\left(g^{-1}(E \cos^2\theta)\right)
\, f(E \cos^2\theta)
}
\, d\theta, \\[8pt]
I_2(E)
&=
\int_{-\pi/2}^{\pi/2}
\frac{
g''\!\left(g^{-1}(E \cos^2\theta)\right)
\, \cos^2\theta
}{
h'\!\left(h^{-1}(\sqrt{E} \sin\theta)\right)
\,
\big(g'\!\left(g^{-1}(E \cos^2\theta)\right)\big)^3
\, f(E \cos^2\theta)
}
\, d\theta, \\[8pt]
I_3(E)
&=
\int_{-\pi/2}^{\pi/2}
\frac{
f'\!(E \cos^2\theta)
\, \cos^2\theta
}{
h'\!\left(h^{-1}(\sqrt{E} \sin\theta)\right)
\, g'\!\left(g^{-1}(E \cos^2\theta)\right)
\, \big(f(E \cos^2\theta)\big)^2
}
\, d\theta.
\end{align*}

Next, we manipulate \(I_1(E)\) to introduce the factor \(\cos^2\theta\) into the integral.  
Substituting \(x = h^{-1}(\sqrt{E} \sin\theta)\) and \(z = E - F(x)\) into \(I_1(E)\), we have
\[
I_1(E)
= \frac{1}{2\sqrt{E}}
\int_{-\pi/2}^{\pi/2}
\frac{h''(x)\,\sin\theta \, d\theta}
{(h'(x))^3 \, g'\!\big(g^{-1}(z)\big)\, f(z)}.
\]

We now integrate \(I_1(E)\) by parts:
\begin{align*}
I_1(E)
&= -\frac{1}{2\sqrt{E}}
\left[
\frac{h''(x)\,\cos\theta}
{(h'(x))^3 \, g'\!\big(g^{-1}(z)\big)\, f(z)}
\right]_{-\pi/2}^{\pi/2}
\\
&\quad
+ \frac{1}{2\sqrt{E}}
\int_{-\pi/2}^{\pi/2}
\frac{d}{d\theta}
\!\left[
\frac{h''(x)}
{(h'(x))^3 \, g'\!\big(g^{-1}(z)\big)\, f(z)}
\right]
\cos\theta \, d\theta.
\end{align*}
The boundary term vanishes since \(\cos(\pm \pi/2)=0\) and all functions are smooth.  
To compute the derivative with respect to \(\theta\), we note that
\[
\frac{d\square }{d\theta}
= \frac{d\square }{dx}\,\frac{dx}{d\theta}
\quad\text{and}\quad
\frac{dx}{d\theta}
= \frac{\sqrt{E}\cos\theta}{h'(x)}.
\]

Define
\[
B =
\frac{h''(x)}
{(h'(x))^3 \, g'\!\big(g^{-1}(z)\big)\, f(z)},
\]
so that
\[
\frac{dB}{d\theta}
=
\frac{
h''' \frac{dx}{d\theta} (h')^3 g'\!\big(g^{-1}(z)\big) f(z)
- h'' \frac{d}{d\theta}\!\left[(h')^3 g'\!\big(g^{-1}(z)\big) f(z)\right]
}{
\left[(h')^3 g'\!\big(g^{-1}(z)\big) f(z)\right]^2
}.
\]

The derivative in the numerator expands as
\begin{align*}
\frac{d}{d\theta}\!\left[(h')^3 g'\!\big(g^{-1}(z)\big) f(z)\right]
&= 3(h')^2 h'' \frac{dx}{d\theta} g'\!\big(g^{-1}(z)\big) f(z)
\\
&\quad
+ (h')^3
\left[
\frac{g''\!\big(g^{-1}(z)\big)(-F')\,\frac{dx}{d\theta}}{g'\!\big(g^{-1}(z)\big)} f(z)
+ g'\!\big(g^{-1}(z)\big) f'(z)(-F')\,\frac{dx}{d\theta}
\right]
\\[2pt]
&= \frac{dx}{d\theta}{\;}\frac{1}{g'\!\big(g^{-1}(z)\big)}
\Big\{
3(h')^2 h'' \big[g'\!\big(g^{-1}(z)\big)\big]^2 f(z)
- (h')^3 g''\!\big(g^{-1}(z)\big) F' f(z)
\\
&\qquad
- (h')^3 \big[g'\!\big(g^{-1}(z)\big)\big]^2 f'(z) F'
\Big\}.
\end{align*}

Substituting back, we find
\begin{align*}
\frac{dB}{d\theta}
&=
\frac{\frac{dx}{d\theta}}
{(h')^6 \big(g'\!\big(g^{-1}(z)\big)\big)^3 (f(z))^2}
\Big\{
h'''(h')^3 \big[g'\!\big(g^{-1}(z)\big)\big]^2 f(z)
- 3(h'')^2 (h')^2 \big[g'\!\big(g^{-1}(z)\big)\big]^2 f(z)
\\
&\qquad
+ h''(h')^3 g''\!\big(g^{-1}(z)\big) F' f(z)
+ h''(h')^3 \big[g'\!\big(g^{-1}(z)\big)\big]^2 F' f'(z)
\Big\}.
\end{align*}

Therefore,
\begin{align*}
I_1(E)
&= \frac{1}{2}
\int_{-\pi/2}^{\pi/2}
\frac{\cos^2\theta\, d\theta}
{(h')^5 \big(g'\!\big(g^{-1}(z)\big)\big)^3 (f(z))^2}
\Big\{
h''' h' \big[g'\!\big(g^{-1}(z)\big)\big]^2 f(z)
- 3(h'')^2 \big[g'\!\big(g^{-1}(z)\big)\big]^2 f(z)
\\
&\qquad
+ h'' h' g''\!\big(g^{-1}(z)\big) F' f(z)
+ h'' h' \big[g'\!\big(g^{-1}(z)\big)\big]^2 F' f'(z)
\Big\}.
\end{align*}

Next, we rewrite \(I_2(E)\) and \(I_3(E)\) so that all three integrals share a common denominator:
\begin{align*}
I_2(E)
&= \int_{-\pi/2}^{\pi/2}
\frac{g''\!\big(g^{-1}(z)\big)\cos^2\theta}
{h'(x)\,\big(g'\!\big(g^{-1}(z)\big)\big)^3\,f(z)}\, d\theta
\\[3pt]
&= \int_{-\pi/2}^{\pi/2}
\frac{(h')^4\, g''\!\big(g^{-1}(z)\big)\, f(z)\, \cos^2\theta}
{(h')^5\, \big(g'\!\big(g^{-1}(z)\big)\big)^3\, (f(z))^2}\, d\theta,
\\[8pt]
I_3(E)
&= \int_{-\pi/2}^{\pi/2}
\frac{f'(z)\cos^2\theta}
{h'(x)\, g'\!\big(g^{-1}(z)\big)\, (f(z))^2}\, d\theta
\\[3pt]
&= \int_{-\pi/2}^{\pi/2}
\frac{(h')^4\, \big(g'\!\big(g^{-1}(z)\big)\big)^2\, f'(z)\, \cos^2\theta}
{(h')^5\, \big(g'\!\big(g^{-1}(z)\big)\big)^3\, (f(z))^2}\, d\theta.
\end{align*}

Combining all three integrals, we obtain
\begin{align*}
\frac{dT}{dE}
&= -\frac{2}{\sqrt{2}}\big(I_1(E) + I_2(E) + I_3(E)\big)
\\[3pt]
&= -\frac{1}{\sqrt{2}} \int_{-\pi/2}^{\pi/2}
\Big\{
h'''h'\,\big[g'\!\big(g^{-1}(z)\big)\big]^2 f(z)
- 3(h'')^2\,\big[g'\!\big(g^{-1}(z)\big)\big]^2 f(z)
\\
&\quad
+ h''h'\, g''\!\big(g^{-1}(z)\big) F' f(z)
+ h''h'\,\big[g'\!\big(g^{-1}(z)\big)\big]^2 F' f'(z)
\\
&\quad
+ 2(h')^4\, g''\!\big(g^{-1}(z)\big) f(z)
+ 2(h')^4\, \big[g'\!\big(g^{-1}(z)\big)\big]^2 f'(z)\\
&\quad
\Big\}
\frac{\cos^2\theta}
{(h')^5\, \big(g'\!\big(g^{-1}(z)\big)\big)^3\, (f(z))^2}\, d\theta.
\end{align*}

Rearranging the terms for later convenience gives
\begin{align*}
\frac{dT}{dE}
&= \frac{1}{\sqrt{2}} \int_{-\pi/2}^{\pi/2}
\Big\{
3(h'')^2\,\big[g'\!\big(g^{-1}(z)\big)\big]^2 f(z)
- h'''h'\,\big[g'\!\big(g^{-1}(z)\big)\big]^2 f(z)
\\
&\qquad
- 3h''h'\, g''\!\big(g^{-1}(z)\big) F' f(z)
- h''h'\,\big[g'\!\big(g^{-1}(z)\big)\big]^2 F' f'(z)
\\
&\qquad
- 2(h')^4\, g''\!\big(g^{-1}(z)\big) f(z)
- 2(h')^4\, \big[g'\!\big(g^{-1}(z)\big)\big]^2 f'(z)\\
&\qquad
\Big\}
\frac{\cos^2\theta}
{(h')^5\, \big(g'\!\big(g^{-1}(z)\big)\big)^3\, (f(z))^2}\, d\theta.
\end{align*}

Thus, the sign of \( \frac{dT}{dE} \) is determined solely by the expression
\[
K = \big[3(h'')^2 - h'''h'\big]\big(g'(g^{-1}(z))\big)^2 f(z) 
    - \big[h''h'F'(x) + 2(h')^4\big]\Big[\big(g'(g^{-1}(z))\big)^2 f'(z) + g''(g^{-1}(z)) f(z)\Big].
\]

Since \(h^2 = F\), we compute the derivatives:
\[
h' = \frac{F'}{2h}, \qquad 
h'' = \frac{2F''F - (F')^2}{4hF}, \qquad 
h''' = \frac{4F^2F''' - 6FF'F'' + 3(F')^3}{8hF^2}.
\]
These yield the identities
\[
3(h'')^2 - h'''h' = \frac{6F(F'')^2 - 3(F')^2F'' - 2FF'F'''}{8F^2},
\qquad
h''h'F' + 2(h')^4 = \frac{2F(F')^2F''}{8F^2}.
\]

From the definitions of \(g\) and \(f\) (see~\eqref{eq_1}–\eqref{eq_2}), we have
\[
g' = \frac{G'}{y}, 
\qquad 
f^2 = \frac{y^2}{2G}.
\]
Hence,
\[
(g')^2 f^2 = \frac{(G')^2}{2G}, 
\qquad 
(g')^2 f f' = \frac{2Gg' - (G')^2}{4G^2}, 
\qquad 
g'' f^2 = \frac{G'' - g'}{2G}.
\]
Combining these relations gives
\[
(g'(g^{-1}))^2 f f' + g''(g^{-1}) f^2 
  = \frac{2Gg' - (G')^2}{4G^2} + \frac{G'' - g'}{2G} 
  = \frac{2G G'' - (G')^2}{4G^2}.
\]

Now we compute \(K \cdot f\):
\begin{align*}
K \cdot f 
&= \big[3(h'')^2 - h'''h'\big]\big(g'(g^{-1})\big)^2 f^2 
 - \big[h''h'F' + 2(h')^4\big]\big[\big(g'(g^{-1})\big)^2 f f' + g''(g^{-1})f^2\big] \\[4pt]
&= \big[3(h'')^2 - h'''h'\big]\frac{(G')^2}{2G} 
 - \big[h''h'F' + 2(h')^4\big]\frac{2G G'' - (G')^2}{4G^2} \\[4pt]
&= \left[\frac{6F(F'')^2 - 3(F')^2F'' - 2FF'F'''}{8F^2}\right] \frac{G(G')^2}{2G^2} 
 - \left[\frac{2F(F')^2F''}{8F^2}\right]\frac{2G G'' - (G')^2}{4G^2} \\[4pt]
&= \frac{\big(6F(F'')^2 - 3(F')^2F'' - 2FF'F'''\big)G(G')^2 
    - F(F')^2F''\big(2G G'' - (G')^2\big)}{16F^2G^2}.
\end{align*}

Therefore, the sign of \( \frac{dT}{dE} \) depends exclusively on the expression
\[
M= \big(6F(F'')^2 - 3(F')^2F'' - 2FF'F'''\big)G(G')^2 
  + F(F')^2F''\big((G')^2 - 2G G''\big).
\]

Consequently, \( \frac{dT}{dE} \geq 0 \) on \( (0, E_0) \) if \( M(x, y) \geq 0 \) for all \( (x, y) \in \Omega(E_0) \),
and \( \frac{dT}{dE} \leq 0 \) on \( (0, E_0) \) if \( M(x, y) \leq 0 \) for all \( (x, y) \in \Omega(E_0) \).

\medskip
This completes the proof of Theorem~\ref{th1}.

\end{proof}

\section{Applications}

In this section, we illustrate the applicability of Theorem~\ref{th1} through some examples, beginning with a representative case that generalizes many classical families of polynomial Hamiltonian systems.

\begin{example}\label{ex:general}
Consider the Hamiltonian system
\begin{equation}
\begin{cases}
x'_1 = x_2\big(b_1 + b_2 x_2^2\big), \\[4pt]
x'_2 = -x_1\big(a_1 + a_2 x_1 + a_3 x_1^2\big),
\end{cases}
\end{equation}
where \(a_1 > 0\) and \(b_1 > 0\).
By performing the change of variables
\begin{equation}\label{nor}
x = x_1, \qquad y = \sqrt{\frac{b_1}{a_1}}\,x_2, \qquad d\tau = \sqrt{a_1 b_1}\,dt,
\end{equation}
we obtain the normalized system
\begin{equation}\label{sist_1}
\begin{cases}
x' = y\big(1 + c y^2\big), \\[2mm]
y' = -x\big(1 + a x + b x^2\big),
\end{cases}
\end{equation}
where
\[
a = \frac{a_2}{a_1}, \qquad b = \frac{a_3}{a_1}, \qquad c = \frac{a_1 b_2}{b_1^2}.
\]
Next, we study the monotonicity of the period function \(T(E)\) associated with the center at \((0,0)\) of the Hamiltonian system~\eqref{sist_1}, in terms of the parameters \(a\), \(b\), and \(c\).
\end{example}

The Hamiltonian function corresponding to system~\eqref{sist_1} is
\[
H(x,y) = \frac{1}{2}x^2 + \frac{a}{3}x^3 + \frac{b}{4}x^4 
        + \frac{1}{2}y^2 + \frac{c}{4}y^4.
\]
Hence, we identify
\[
F(x) = \frac{1}{2}x^2 + \frac{a}{3}x^3 + \frac{b}{4}x^4,
\qquad
G(y) = \frac{1}{2}y^2 + \frac{c}{4}y^4.
\]

After a straightforward but lengthy computation, we obtain
\[
M(x,y) = \frac{1}{24}x^4y^4\big(A(x)P(y) - B(x)Q(y)\big),
\]
where
\begin{align*}
A(x) &= 10a^2 - 9b + a(30b + 4a^2)x + b(36b + 16a^2)x^2 + 24ab^2x^3 + 9b^3x^4, \\[2pt]
B(x) &= (6 + 4ax + 3bx^2)(1 + ax + bx^2)^2(1 + 2ax + 3bx^2), \\[2pt]
P(y) &= (1 + cy^2)^2(2 + cy^2), \\[2pt]
Q(y) &= c(3 + cy^2).
\end{align*}

This factorization shows that the sign of \(M(x,y)\) is governed by
\[
N(x,y) = A(x)P(y) - B(x)Q(y),
\]
since \(x^4y^4/24 > 0\) for all \((x,y)\neq(0,0)\) in the period annulus.

Let \(\Delta = a^2 - 4b\) and \(x_{1,2} = \frac{-a \pm \sqrt{\Delta}}{2b}\) be the real roots of \(p(x) = 1 + ax + bx^2\), when they exist.  
We now describe the behavior of \(T(E)\) for different parameter configurations.

\begin{theorem}\label{th_2}
Let \(T(E)\) denote the period function associated with the center at \((0,0)\) of system~\eqref{sist_1}. The following statements hold:
\begin{enumerate}[label=(\Alph*)]
\item If \(a = b = c = 0\), then \(T(E)\) is constant for all \(E\).
\item If \(a = b = 0\) and \(c > 0\), then \(T(E)\) is decreasing for all \(E > 0\).
\item Suppose \(a \neq 0\), \(b = 0\), and \(c \ge 0\). Let \(x_0 \neq -1/a\) satisfy \(F(x_0) = F(-1/a) =: E_0\). Then \(T(E)\) is increasing on \((0, E_0)\) for every \(0 < c \le \frac{2A(x_0)}{3B(x_0)}\).
\item Suppose \(b \neq 0\) and \(\Delta \ge 0\). Let \(x_0\) be the real root of \(p(x)\) defining the boundary of the period annulus, and let \(r_0 \neq x_0\) satisfy \(F(r_0) = F(x_0) =: E_0\). Denote by \(r\) the point in the annulus where \(f(x) := \frac{A(x)}{B(x)}\) attains its absolute minimum. Then \(T(E)\) is increasing on \((0, E_0)\) for every \(0 \le c \le \frac{2A(r)}{3B(r)}\).
\item Suppose \(\Delta < 0\), \(a = 0\), and \(c \ge 0\).
  \begin{enumerate}[label=(\roman*)]
  \item If \(c = 0\), then \(T(E)\) is decreasing for all \(0 < E \le E_0 := F(x_0)\), where \(x_0 = \sqrt{(\sqrt{5}-2)/2}\).
  \item If \(c > 0\), let \(x_m > x_0\) be an absolute maximum of \(f(x) = \frac{A(x)}{B(x)}\), and define \(y_m = y_m(c)\) by \(G(y_m) = F(x_m) =: E_0\). Then there exists a unique \(c_0 > 0\) satisfying
  \[
  \frac{A(x_m)}{B(x_m)} = \frac{Q(y_m)}{P(y_m)},
  \]
  such that \(T(E)\) is decreasing on \((0, E_0)\) for every \(c \ge c_0\).
  \end{enumerate}
\item Suppose \(a \neq 0\), \(b > \frac{a^2}{4}\), and \(c = 0\).
  \begin{enumerate}[label=(\roman*)]
  \item If \(\frac{a^2}{4} < b < \frac{10}{9}a^2\), then \(A(x)\) has two real roots of the same sign (\(x_2 < x_1 < 0\) if \(a > 0\), or \(0 < x_1 < x_2\) if \(a < 0\)). In this case, \(T(E)\) is increasing for all \(0 < E \le E_0 := F(x_1)\).
  \item If \(b > \frac{10}{9}a^2\), then \(A(x)\) has two real roots with opposite signs (\(x_2 < 0 < x_1\)). In this case, \(T(E)\) is decreasing for all \(0 < E \le E_0 := F(x_1)\).
  \end{enumerate}
\item Suppose \(a \neq 0\), \(b > \frac{a^2}{4}\), and \(c > 0\).
  \begin{enumerate}[label=(\roman*)]
  \item If \(\frac{a^2}{4} < b \le \frac{a^2}{3}\), then \(A(x)\) has two distinct roots (\(x_2 < x_1 < 0\) if \(a > 0\), or \(0 < x_1 < x_2\) if \(a < 0\)). Let \(x_0 \neq x_1\) satisfy \(F(x_0) = F(x_1) =: E_0\). Then \(T(E)\) is increasing on \((0, E_0)\) for every \(0 < c \le c_0 := \frac{2A(x_0)}{3B(x_0)}\).
  \item If \(b > \frac{a^2}{3}\) and \(c > 0\), let \(x_m\) be a maximum point of \(f(x) = \frac{A(x)}{B(x)}\), and define \(y_m = y_m(c)\) by \(G(y_m) = F(x_m)\). Then there exists \(c_0 > 0\) determined implicitly by
  \[
  \frac{A(x_m)}{B(x_m)} = \frac{Q(y_m)}{P(y_m)},
  \]
  such that \(T(E)\) is decreasing on \((0, E_0)\) for every \(c\geq c_0\).
  \end{enumerate}
\end{enumerate}
\end{theorem}

\begin{proof}
The proof consists in analyzing the sign of \(N(x,y)\) in each case. 

\subsection*{Case (A): \(a = b = c = 0\)} This case corresponds to a linear isochronous center; hence \(T(E)\) is constant for all \(E\). 
We proceed to analyze the remaining cases.

\subsection*{Case (B): \(a = b = 0\) and \(c > 0\)}

In this case, the origin \((0,0)\) is a global (unbounded) center. We establish the following result.

\begin{lemma}[Case (B)]
If \(a = b = 0\) and \(c > 0\), then \(M(x,y) \leq 0\) for all \((x,y) \in \mathbb{R}^2\).
\end{lemma}

\begin{proof}
When \(a = b = 0\), we have \(A(x) = 0\) and \(B(x) = 6\) for all \(x \in \mathbb{R}\). 
Therefore,
\[
N(x,y) = A(x)P(y) - B(x)Q(y) = -6Q(y) \leq 0,
\]
since \(Q(y) > 0\) for all \(y\) whenever \(c > 0\). 
Hence \(M(x,y) \leq 0\) for all \((x,y) \in \mathbb{R}^2\), and by Theorem~\ref{th1} we conclude that \(T(E)\) is decreasing for all \(E > 0\).
\end{proof}

\subsection*{Case (C): \(a \neq 0\), \(b = 0\), and \(c \geq 0\)}

In this case, the origin \((0,0)\) is a bounded center contained in the region 
\[
\Omega_a(E_0) = \{(x,y) : 0 \leq H(x,y) \leq E_0\},
\]
where \(E_0 = F(x_0)\) and \(x_0 = -\frac{1}{a}\) is the root of \(1 + a x = 0\) (see Fig.~\ref{Fig2a}).

\

\begin{figure}[hbt!]
    \centering
    \begin{subfigure}[b]{0.45\textwidth}
        \centering
        \includegraphics[width=.65\linewidth]{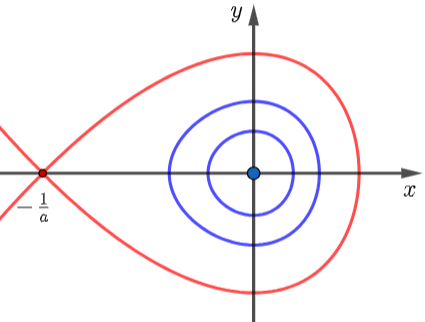}
        \caption{\(\Omega_a(E_0)\) for \(a > 0.\)}
        \label{Fig2a}
    \end{subfigure}
    \hfill
    \begin{subfigure}[b]{0.45\textwidth}
        \centering
        \includegraphics[width=.65\linewidth]{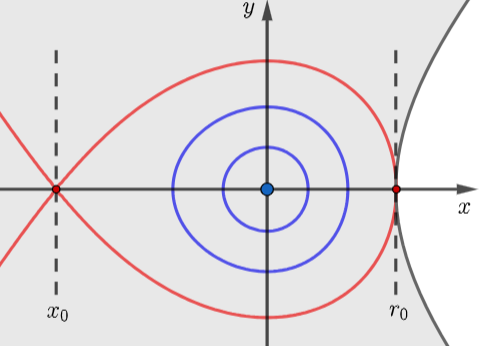}
        \caption{\(M(x,y) \geq 0\) in \(\Omega_a(E_0)\).}
        \label{Fig2b}
    \end{subfigure}
    \caption{Illustration of the bounded center corresponding to Case (C).}
    \label{Fig2}
\end{figure}

We have the following result.

\begin{lemma}[Case (C)]
Suppose \(a \neq 0\), \(b = 0\), and \(c \geq 0\). Let \(r_0 \neq x_0 = -\frac{1}{a}\) be such that \(F(r_0) = F(x_0)\). 
Then \(M(x,y) \geq 0\) in \(\Omega_a(E_0)\) for all \(0 \leq c \leq c_0 = \dfrac{2A(r_0)}{3B(r_0)}\) (see Fig~\ref{Fig2b}).
\end{lemma}

\begin{proof}
In this case,
\[
A(x) = a^2(10 + 4a x)
\quad \text{and} \quad 
B(x) = (6 + 4a x)(1 + a x)^2(1 + 2a x).
\]
Since \(P(y) > 0\) for all \(y\), it follows that
\[
N(x,y) \geq 0 
\;\Leftrightarrow\;
P(y)\!\left(A(x) - B(x)\frac{Q(y)}{P(y)}\right) \geq 0
\;\Leftrightarrow\;
A(x) - B(x)\frac{Q(y)}{P(y)} \geq 0.
\]

Note that the function \(\frac{Q(y)}{P(y)}\) is strictly decreasing and attains its maximum at \(y = 0\). 
Hence,
\[
\frac{Q(y)}{P(y)} \leq \frac{Q(0)}{P(0)} = \frac{3c}{2}.
\]

Define the function \(f(x) = \dfrac{A(x)}{B(x)}\) on the interval \((x_0, r_0]\). 
Since \(f\) attains its global minimum at \(x = r_0\), we have
\[
\frac{A(r_0)}{B(r_0)} \leq \frac{A(x)}{B(x)}, \quad \forall\, x \in (x_0, r_0].
\]
Therefore,
\[
A(x) - B(x)\frac{Q(y)}{P(y)} 
\geq A(x) - \frac{3c}{2}B(x)
\geq A(x) - \frac{3c_0}{2}B(x)
\geq A(x) - \frac{A(r_0)}{B(r_0)}B(x)
\geq 0.
\]

Hence \(N(x,y) \geq 0\) throughout \(\Omega_a(E_0)\) for all \(0 \leq c \leq \frac{2A(r_0)}{3B(r_0)}\). 
By Theorem~\ref{th1}, it follows that \(T(E)\) is increasing for all \(0 < E \leq F(x_0)\).
\end{proof}

\subsection*{Case (D): \(b \neq 0\) and \(\Delta = a^2 - 4b \geq 0\)}

In this case, \(p(x) = 1 + a x + b x^2\) has two real roots, denoted by \(x_1\) and \(x_2\). Consequently, the center at \((0,0)\) is bounded by the homoclinic orbit passing through the point \((x_0,0)\), where \(x_0 \in \{x_1, x_2\}\). 
Without loss of generality, we may choose \(x_0 = x_1\), as illustrated in Figure~\ref{Fig3}.

\begin{figure}[hbt!]
    \centering
    \begin{subfigure}[b]{0.45\textwidth}
        \centering
        \includegraphics[width=.7\linewidth]{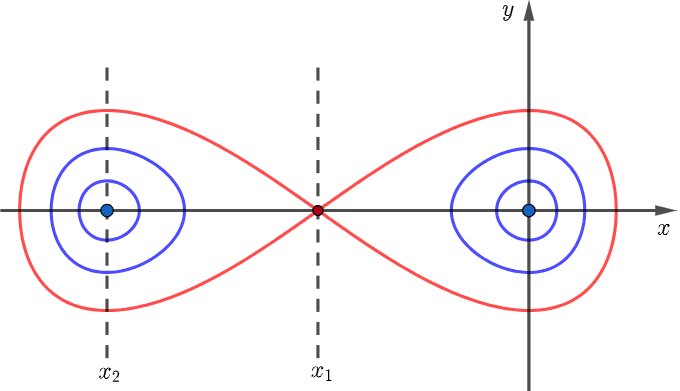}
        \caption{\(x_2 < x_1 < 0\)\, \((a > 0,\, b > 0)\).}
    \end{subfigure}
    \hfill
    \begin{subfigure}[b]{0.45\textwidth}
        \centering
        \includegraphics[width=0.5\linewidth]{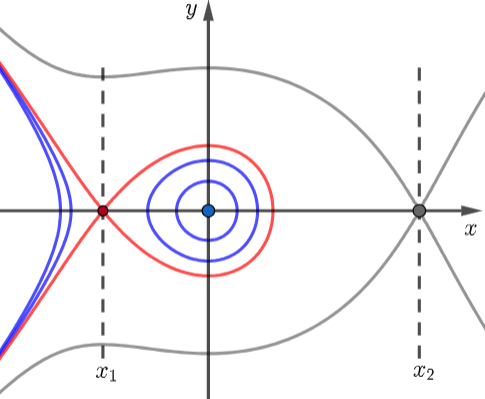}
        \caption{\(x_1 < 0 < x_2\)\, \((a > 0,\, b < 0)\).}
    \end{subfigure}
     
    \vspace{0.5em}
     
    \begin{subfigure}[b]{0.45\textwidth}
        \centering
        \includegraphics[width=.7\linewidth]{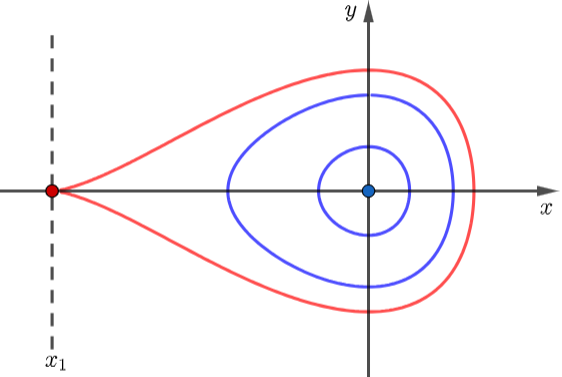}
        \caption{\(x_1 = x_2\)\, \((a > 0,\, b = a^2/4)\).}
    \end{subfigure}
    \hfill
    \begin{subfigure}[b]{0.45\textwidth}
        \centering
        \includegraphics[width=.6\linewidth]{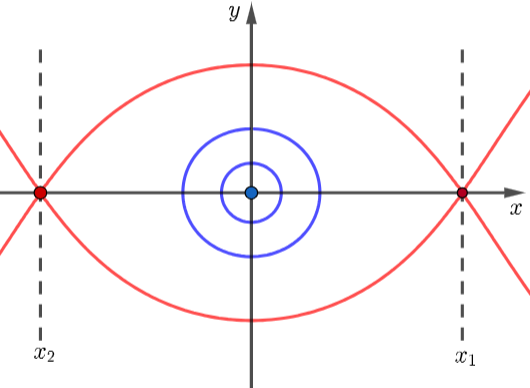}
        \caption{\(x_2 < 0,\, x_1 = -x_2\)\, \((a = 0,\, b < 0)\).}
    \end{subfigure}

    \caption{Illustration of four among the seven possible phase configurations.}
    \label{Fig3}
\end{figure}

Thus, studying the monotonicity of the period function \(T(E)\) associated with the center at \((0,0)\) for \(\Delta \geq 0\) reduces to determining the sign of \(N(x,y)\) in the region
\[
\Omega(E_0) = \{(x,y) \in \mathbb{R}^2 : 0 \leq H(x,y) \leq F(x_0)\},
\]
where \(x_0\) denotes the root of \(1 + a x + b x^2 = 0\) that defines the boundary of the period annulus.

\begin{lemma}[Case (D)]
Suppose \(b \neq 0\) and \(\Delta \geq 0\). Let \(x_0\) be the root of \(1 + a x + b x^2 = 0\) that defines the boundary of the period annulus, and let \(r_0 \neq x_0\) be such that \(F(r_0) = F(x_0)\). 
Then \(N(x,y) \geq 0\) in \(\Omega(E_0)\) for all \(0 \leq c \leq c_0 = \frac{2A(r)}{3B(r)}\), where \(r\) is the point at which the function \(f(x) = \frac{A(x)}{B(x)}\) attains its absolute minimum on the interval determined by the period annulus.
\end{lemma}

\begin{proof}
The argument follows similarly to that of Case~(C). Without loss of generality, assume that \(x_0 < 0\). Since \(P(y) > 0\) for all \(y\), we have
\[
N(x,y) \geq 0 
\;\Leftrightarrow\;
P(y)\!\left(A(x) - B(x)\frac{Q(y)}{P(y)}\right) \geq 0
\;\Leftrightarrow\;
A(x) - B(x)\frac{Q(y)}{P(y)} \geq 0.
\]

Note that the function \(\frac{Q(y)}{P(y)}\) is strictly decreasing and attains its maximum at \(y = 0\). Hence,
\[
\frac{Q(y)}{P(y)} \leq \frac{Q(0)}{P(0)} = \frac{3c}{2}.
\]

Let \(r \in [x_0, r_0]\) denote the point where the function \(\frac{A(x)}{B(x)}\) reaches its absolute minimum. Then,
\[
\frac{A(r)}{B(r)} \leq \frac{A(x)}{B(x)}, \quad \forall\, x \in [x_0, r_0].
\]
Therefore,
\[
A(x) - B(x)\frac{Q(y)}{P(y)} 
\geq A(x) - \frac{3c}{2}B(x)
\geq A(x) - \frac{3c_0}{2}B(x)
= A(x) - \frac{A(r)}{B(r)}B(x)
\geq 0.
\]

Thus, \(N(x,y) \geq 0\) in \(\Omega(E_0)\) for all \(0 \leq c \leq \frac{2A(r)}{3B(r)}\). 
In particular, when \(b > 0\), we have \(r = r_0\); that is, the global minimum of the function \(\frac{A(x)}{B(x)}\) occurs at \(r_0\). 
Consequently, the period function \(T(E)\) is increasing for all \(0 \leq E \leq E_0\).
\end{proof}

\begin{remark}\label{rmk:cases}
The following observations (1)--(5) apply to cases (C) and (D).

\begin{enumerate}
\item Let \(c_1 = \dfrac{2A(0)}{3B(0)}\).  
For each fixed \(c \in (c_0, c_1)\), there exists \(x_c \in (x_0, r_0)\) such that 
\[
N(x,y) \geq 0, \quad \text{for all } (x,y) \text{ satisfying } 0 \leq H(x,y) \leq F(x_c)
\text{ (see Fig.~\ref{Fig4a})}.
\]
In these cases, numerical simulations indicate that the period function \(T(E)\) reaches a maximum at some
\[
F(x_c) < E^* < F(x_0),
\]
and then becomes decreasing on the interval \((E^*, F(x_0))\).

\item For each fixed \(c > c_1\), there exists \(x_c \in (x_0, r_0)\) such that 
\[
N(x,y) \leq 0, \quad \text{for all } (x,y) \text{ satisfying } 0 \leq H(x,y) \leq F(x_c)
\text{ (see Fig.~\ref{Fig4b})}.
\]
Consequently, \(T(E)\) is decreasing on \((0, F(x_c)]\).  
Numerical plots indicate that \(T(E)\) attains a minimum at some 
\[
F(x_c) < E^* < F(x_0),
\]
and becomes increasing on the interval \((E^*, F(x_0))\).

\item In the special case \(a = 0\) and \(b < 0\), one has \(c_0 = c_1\).  
Hence, \(T(E)\) is increasing on \((0, F(x_0)]\) for all \(0 \leq c \leq c_0\),  
and decreasing on \((0, F(x_c)]\) for each fixed \(c > c_0\) (see Fig.~\ref{Fig4c}).

\item When \(a \neq 0\) and \(c = c_0\), the sign of \(N(x,y)\) is not well defined in a neighborhood of \((0,0)\).  
Therefore, no definite conclusion can be drawn regarding the monotonicity of the period function for orbits near the origin (see Fig.~\ref{Fig4d}).

\item In the cases \(b = 0\) (case C) or \(b > 0\), the point \(x_c\) can be determined as the solution of the equation
\[
\sigma(x_c) = 0,
\quad \text{where} \quad
\sigma(x) = 2A(x) - 3cB(x).
\]
\end{enumerate}

\begin{figure}[hbt!]
    \centering
    \begin{subfigure}[b]{0.45\textwidth}
        \centering
        \includegraphics[width=0.7\linewidth]{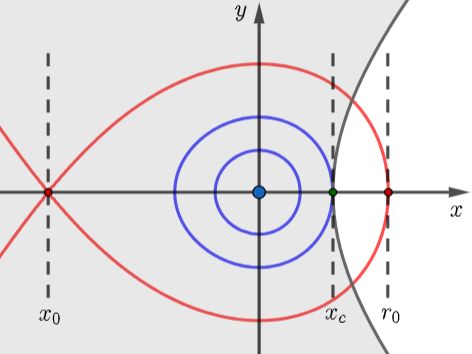}
        \caption{}
        \label{Fig4a}
    \end{subfigure}
    \hfill
    \begin{subfigure}[b]{0.45\textwidth}
        \centering
        \includegraphics[width=0.7\linewidth]{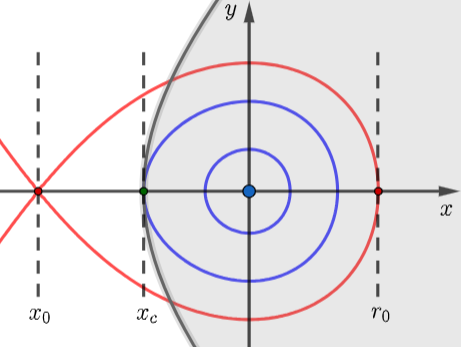}
        \caption{}
        \label{Fig4b}
    \end{subfigure}
   
     \vspace{0.5em}
     
    \begin{subfigure}[b]{0.45\textwidth}
        \centering
        \includegraphics[width=.7\linewidth]{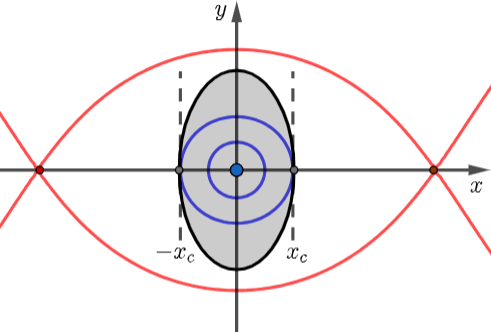}
        \caption{}
        \label{Fig4c}
    \end{subfigure}
    \hfill
    \begin{subfigure}[b]{0.45\textwidth}
        \centering
        \includegraphics[width=.6\linewidth]{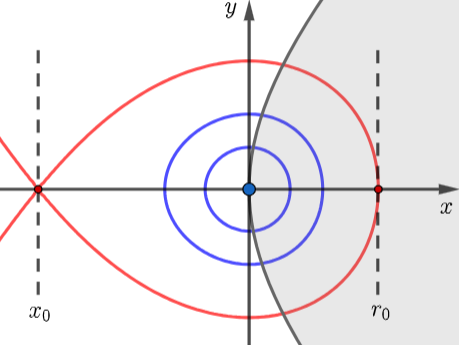}
        \caption{}
        \label{Fig4d}
    \end{subfigure}

    \caption{Analysis of the sign of \(N(x,y)\) in the region 
    \(\{(x,y) : 0 \leq H(x,y) \leq F(x_c)\}\).}
\end{figure}
\end{remark}

In~\cite{freire2004first}, Fig.~2, the authors presented the plot of the period function associated with the Hamiltonian system
\[
H(x,y) = k\left(\frac{x^2}{2} + \frac{x^3}{3}\right) + \frac{y^2}{2} + \frac{y^4}{4},
\]
for different values of \(k\): \(k = 1, 1.17525, 1.5, 2, 5, 10\). 
It was observed that, for \(k = 1\), the period function is increasing. 
When \(k \approx 1.17525\), it exhibits a minimum (almost imperceptible at the scale of the figure, so that the center appears to be isochronous), 
and it becomes decreasing for larger values of \(k\), such as \(k = 5\) and \(k = 10\).

In what follows, we apply the result established for Case~(C) to carry out an analytical study of this example.

After applying the change of variables~\eqref{nor} to the Hamiltonian system defined by \(H\), we obtain the normalized system
\[
x' = y(1 + k y^2), \qquad
y' = -x(1 + x),
\]
which corresponds precisely to case~(C), with \(c=k\).
In this case, \(x_0 = -1\) and \(F(x_0) = \tfrac{1}{6}\).
Solving \(F(r_0) = F(-1) = \tfrac{1}{6}\), we find \(r_0 = \tfrac{1}{2}\).

Using this value, and recalling that
\[
A(x) = 10 + 4x, \qquad
B(x) = (6 + 4x)(1+x)^2(1+2x),
\]
we compute
\[
k_0 = \frac{2A(1/2)}{3B(1/2)} = \frac{2}{9}.
\]
Therefore, according to Lemma~(Case~C), the period function \(T(E)\) is \emph{increasing} on \((0, 1/6)\) for all \(0 < k \leq 2/9\).

Furthermore, since
\[
k_1 = \frac{2A(0)}{3B(0)} = \frac{10}{9},
\]
it follows from Remark~\ref{rmk:cases}(1) that, for each \(k_0 < k < k_1\), there exists \(0 < x_k < \tfrac{1}{2}\) such that \(T(E)\) is increasing on \((0, F(x_k))\).
For instance, for \(k = 1\) we find \(x_1 \approx 0.025147\) and \(F(0.025147) \approx 0.000321\).
Consequently, \(T(E)\) is increasing on \((0, 0.000321)\).
Numerical evidence suggests that \(T(E)\) remains increasing for \(0.000321 < E < 1/6\).

For the limiting case \(k = 10/9\), our analytical result does not determine the monotonicity of \(T(E)\).
However, numerical computations indicate that \(T(E)\) is nearly constant for small energies, suggesting an almost isochronous behavior near the origin.

Finally, for \(k > 10/9\), the period function \(T(E)\) becomes \emph{decreasing} in the interval
\[
0 < E \leq F(x_k) < \tfrac{1}{6}.
\]
In this situation, the period function likely possesses a critical point between \(F(x_k)\) and \(F(1/2) = \tfrac{1}{6}\).
For example, for \(k = 2\) we find \(x_2 \approx -0.12449\) and \(F(-0.12449) \approx 0.007106\), implying that \(T(E)\) is decreasing on \((0, 0.007106)\).
Numerical graphs show the occurrence of a minimum at approximately \(E^* \approx 0.0427\).
A similar behavior is observed for \(k = 10\) and for larger values of \(k\).

These analytical predictions are in complete qualitative agreement with the numerical observations reported in~\cite{freire2004first}.

In the cases where \(\Delta < 0\), the system possesses an unbounded global center at \((0,0)\). We have the following results.

\begin{lemma}[Case E]\label{Lem_E}
Suppose \(\Delta < 0\), \(a = 0\), and \(c \geq 0\). The following statements hold:
\begin{itemize}
    \item[(i)] If \(c = 0\), then
    \[
    N(x,y) \leq 0 \quad \text{in } \Omega(E_0) = \{(x,y) \in \mathbb{R}^2 : 0 \leq H(x,y) \leq E_0\},
    \]
    where \(E_0 = F(x_0)\) and \(x_0 = \sqrt{\tfrac{\sqrt{5} - 2}{2}}\).
    
    \item[(ii)] Suppose \(c > 0\). Let \(x_m > x_0\) be the point of absolute maximum of the function
    \[
    f(x) = \frac{A(x)}{B(x)},
    \]
    and define \(y_m = y_m(c)\) by the condition
    \[
    G(y_m) = F(x_m) := E_0.
    \]
    Then
    \[
    N(x,y) \leq 0 \quad \text{in } \Omega(E_0) = \{(x,y) \in \mathbb{R}^2 : 0 \leq H(x,y) \leq E_0\}
    \]
   for every \(c \geq c_0\), where \(c_0 > 0\) is uniquely determined by the tangency condition (see Fig.~\ref{Fig5})
\[
\frac{A(x_m)}{B(x_m)} = \frac{Q(y_m)}{P(y_m)}.
\]
\end{itemize}
\end{lemma}

\begin{figure}[hbt!]
    \centering
    \includegraphics[width=0.35\linewidth]{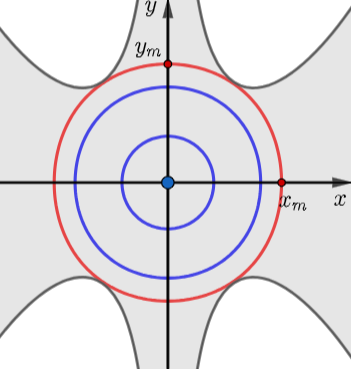}
    \caption{Region \(\Omega(E_0)\) corresponding to Case~E, item~(ii).}
    \label{Fig5}
\end{figure}

\begin{proof}
We work under the hypotheses of the lemma: \(a = 0\), \(b > 0\) (hence \(\Delta = a^2 - 4b < 0\)), and \(c \geq 0\). 
Note that, for \(a = 0\), we have
\[
A(x) = 9b^3 x^4 + 36b^2 x^2 - 9b = 9b\big(b^2x^4 + 4bx^2 - 1\big),
\]
and
\[
B(x) = (6 + 3bx^2)(1 + bx^2)^2(1 + 3bx^2) > 0, \quad \forall\, x \in \mathbb{R}.
\]

\noindent\textbf{(i)} If \(c = 0\), then \(P(y) = 4\) and \(Q(y) = 0\). Hence,
\[
A(x)P(y) - B(x)Q(y) = 4A(x).
\]
Since the real zeros of \(A(x)\) are \(x = \pm x_0\) and \(b > 0\), it follows that \(A(x) \leq 0\) for all \(x \in [-x_0, x_0]\). 
Consequently, \(N(x,y) \leq 0\) in \(\{(x,y) : 0 \leq H(x,y) \leq F(x_0)\}\), which proves (i).

\smallskip

\noindent\textbf{(ii)} Now suppose \(c > 0\). 
Since both \(B(x)\) and \(P(y)\) are positive functions, we can write
\[
N(x,y) = B(x)P(y)\!\left(\frac{A(x)}{B(x)} - \frac{Q(y)}{P(y)}\right) \leq 0
\quad \Longleftrightarrow \quad
\frac{A(x)}{B(x)} - \frac{Q(y)}{P(y)} \leq 0.
\]

Let \(f(x) = \tfrac{A(x)}{B(x)}\). 
Then \(f(x_0) = 0\), \(f'(x_0) = A'(x_0)B(x_0) > 0\), and \(f(x) \to 0\) as \(x \to +\infty\).
Hence, \(f\) attains an absolute maximum at some \(x_m \in (x_0, +\infty)\), with \(f(x_m) > 0\). 

Let \(y_m = y_m(c)\) be the positive solution of \(F(x_m) = G(y_m)\), and define
\[
\varphi(c) = \min_{|y| \leq y_m(c)} \frac{Q(y)}{P(y)}.
\]
The map \(c \mapsto \varphi(c)\) is continuous and positive, since both \(P\) and \(Q\) are positive for \(c > 0\). 
By the Intermediate Value Theorem, there exists \(c_0 > 0\) such that \(\varphi(c_0) = f(x_m)\). 
Therefore,
\[
\frac{A(x_m)}{B(x_m)} = \frac{Q(y_m(c_0))}{P(y_m(c_0))}.
\]

For all \(c \geq c_0\), we have \(\varphi(c) \geq f(x_m)\), and consequently
\[
\frac{A(x)}{B(x)} \leq f(x_m) \leq \varphi(c) \leq \frac{Q(y)}{P(y)},
\]
for all \((x,y)\) such that \(H(x,y) \leq F(x_m)\).

Therefore, \(N(x,y) \leq 0\) throughout the region \(\{0 \leq H(x,y) \leq F(x_m)\}\), which completes the proof.
\end{proof}

\begin{lemma}(Case F)
Suppose \(a \neq 0\), \(b > \dfrac{a^2}{4}\), and \(c = 0\). Consider the polynomial
\[
A(x) = 10a^2 - 9b + a(30b + 4a^2)x + b(36b + 16a^2)x^2 + 24ab^2x^3 + 9b^3x^4.
\]
Then the following statements hold:

\medskip

\noindent\textbf{(i)} If \(a > 0\) and \(\dfrac{a^2}{4} < b < \dfrac{10}{9}a^2\), then \(A(x)\) has two negative real roots \(x_2 < x_1 < 0\). 
In this case, \(A(x) \geq 0\) on \([x_1, 0]\), and consequently,
\[
N(x,y) \geq 0 \quad \text{in} \quad \{(x,y) : 0 \leq H(x,y) \leq F(x_1)\}.
\]

\smallskip

\noindent\textbf{(ii)} If \(a < 0\) and \(\dfrac{a^2}{4} < b < \dfrac{10}{9}a^2\), then \(A(x)\) has two positive real roots \(0 < x_1 < x_2\). 
In this case, \(A(x) \geq 0\) on \([0, x_1]\), and consequently,
\[
N(x,y) \geq 0 \quad \text{in} \quad \{(x,y) : 0 \leq H(x,y) \leq F(x_1)\}.
\]

\smallskip

\noindent\textbf{(iii)} If \(a \neq 0\) and \(b > \dfrac{10}{9}a^2\), then \(A(x)\) has two real roots \(x_2 < 0 < x_1\). 
In this case, \(A(x) \leq 0\) on \([x_2, x_1]\), and consequently,
\[
N(x,y) \leq 0 \quad \text{in} \quad \{(x,y) : 0 \leq H(x,y) \leq F(x_1)\}.
\]

\smallskip

\noindent\textbf{(iv)} If \(a > 0\) and \(b = \dfrac{10}{9}a^2\), then \(A(0) = 0\), with \(A(x) < 0\) for \(x < 0\) and \(A(x) > 0\) for \(x > 0\). 
Thus, the sign of \(N(x,y)\) is not well-defined in a neighborhood of \((0,0)\).
\end{lemma}

\begin{proof}
For \(c = 0\), we have \(P(y) = 2\) and \(Q(y) = 0\). Therefore,
\[
A(x)P(y) - B(x)Q(y) = 4A(x).
\]
Thus, for \((x,y) \neq (0,0)\), the sign of \(N(x,y)\) coincides with the sign of \(A(x)\). Hence, it suffices to study the sign of the polynomial
\[
A(x) = 10a^2 - 9b + a(30b + 4a^2)x + b(36b + 16a^2)x^2 + 24ab^2x^3 + 9b^3x^4.
\]

Note that the leading coefficient is \(9b^3 > 0\); therefore,
\(\lim_{x \to \pm\infty} A(x) = +\infty\).
Hence, \(A(x)\) takes positive values at both ends and can only change sign an even number of times. We also have
\[
A(0) = 10a^2 - 9b, \qquad A'(0) = a(30b + 4a^2).
\]

\medskip
\noindent\textbf{Case (i).} If \(a > 0\) and \(\dfrac{a^2}{4} < b < \dfrac{10}{9}a^2\), then \(A(0) > 0\) and \(A'(0) > 0\). By the behavior at infinity and continuity, \(A(x)\) must possess the necessary critical points to guarantee two intersections with the \(x\)-axis to the left of the origin. Consequently, there exist exactly two negative real roots \(x_2 < x_1 < 0\). Therefore, \(A(x) \geq 0\) on \([x_1, 0]\).

\medskip
\noindent\textbf{Case (ii).} If \(a < 0\) and \(\dfrac{a^2}{4} < b < \dfrac{10}{9}a^2\), the same reasoning applies with the orientation of \(x\) reversed, yielding two positive real roots \(0 < x_1 < x_2\). Therefore, \(A(x) \geq 0\) on \([0, x_1]\).

\medskip
\noindent\textbf{Case (iii).} If \(b > \dfrac{10}{9}a^2\), then \(A(0) = 10a^2 - 9b < 0\). Since \(A(x) \to +\infty\) as \(x \to \pm\infty\), continuity ensures the existence of one root to the left and another to the right of the origin, that is, \(x_2 < 0 < x_1\). Thus, \(A(x) \leq 0\) on \([x_2, x_1]\).

\medskip
\noindent\textbf{Critical case.} If \(b = \dfrac{10}{9}a^2\), then \(A(0) = 0\). Since \(A'(0) = a(30b + 4a^2)\) has the same sign as \(a\), it follows that \(A(x) < 0\) for \(x < 0\) near the origin and \(A(x) > 0\) for \(x > 0\) near the origin. Hence, the sign of \(N(x,y)\) changes arbitrarily close to \((0,0)\).
\end{proof}

\begin{lemma}(Case G)
Suppose \(a \neq 0\), \(b > \dfrac{a^2}{4}\), and \(c > 0\). Then the following statements hold.

\medskip
\noindent
\textbf{(i)} Suppose that \( \dfrac{a^2}{4} < b \le \dfrac{a^2}{3}\). 
In this case, \(A(x)\) has two distinct real roots (\(x_2 < x_1 < 0\) if \(a > 0\), and \(0 < x_1 < x_2\) if \(a < 0\)). 
Let \(x_0 \neq x_1\) be such that \(F(x_0) = F(x_1)\). 
Then \(N(x,y) \geq 0\) in 
\[
\{(x,y) : 0 \le H(x,y) \le F(x_1)\}
\]
for all \(0 < c \le c_0 = \dfrac{2A(x_0)}{3B(x_0)}.\)

\medskip
\noindent
\textbf{(ii)} Suppose \(b > \dfrac{a^2}{3}\) and \(c > 0\).  
Let \(x_m\) be the maximum point of the function
\[
f(x) = \frac{A(x)}{B(x)},
\]
and define \(y_m = y_m(c)\) implicitly by \(G(y_m) = F(x_m)\). 
Then
\[
N(x,y) \le 0 \quad \text{in} \quad \{(x,y) : 0 \le H(x,y) \le F(x_m)\},
\]
for all \(c \geq c_0\), where \(c_0 > 0\) is determined by the tangency condition
\[
\frac{A(x_m)}{B(x_m)} = \frac{Q(y_m)}{P(y_m)}.
\]
\end{lemma}

\begin{proof}
The proof of item~(ii) follows analogously to the proof of item~(ii) in Lemma~\ref{Lem_E}.

\smallskip

For \( \dfrac{a^2}{4} < b \le \dfrac{a^2}{3}\), the function \(A(x)\) satisfies \(A(x_1) = 0\) and \(A(x) > 0\) on \([x_0, x_1)\). 
Since \(Q(y) > 0\) for all \(c > 0\), we have
\[
N(x,y) = A(x)Q(y)\!\left(\frac{P(y)}{Q(y)} - \frac{B(x)}{A(x)}\right) \ge 0
\;\Longleftrightarrow\;
\left(\frac{P(y)}{Q(y)} - \frac{B(x)}{A(x)}\right) \ge 0.
\]

The function \(\dfrac{B(x)}{A(x)}\) is decreasing on \([x_0, x_1)\) and attains its absolute maximum at \(x = x_0\). 
On the other hand, \(\dfrac{P(y)}{Q(y)}\) attains its absolute minimum at \(y = 0\). 
Therefore,
\[
N(x,y) \ge 0 
\;\Longleftrightarrow\;
\left(\frac{P(0)}{Q(0)} - \frac{B(x_0)}{A(x_0)}\right) \ge 0
\;\Longleftrightarrow\;
0 < c \le \frac{2A(x_0)}{3B(x_0)}.
\]
\end{proof}

This completes the proof of Theorem~\ref{th_2}.

\end{proof}

\begin{remark}\label{rmk:caseG_ab3}
In the cases where $\Delta < 0$, situations analogous to those described in 
Remarks~\ref{rmk:cases} may also occur. 
Consider, in particular, the case $a = b = 3$, which falls under item~(i) of Case~G. 

Computing the roots of $A(x)$ yields $x_1 = -\tfrac{1}{3}$. 
Since $F(x_0) = F(x_1) = \tfrac{1}{36}$, we obtain $x_0 \approx 0.1958$. 
Hence,
\[
c_0 = \frac{2A(x_0)}{3B(x_0)} \approx 1.798.
\]
It follows that the period function $T(E)$ is \emph{increasing} on the interval $(0,\, 1/36)$.

Numerical plots indicate the existence of at least one critical point for $E > 1/36$. 
For $c_0 < c < c_1 = \dfrac{2A(0)}{3B(0)} = 7$, the function $T(E)$ remains increasing 
on $(0,\, F(x_c))$, where $x_c$ is determined by the condition $\sigma(x_c) = 0$. 
For $c > 7$, $T(E)$ becomes \emph{decreasing} on $(0,\, F(x_c))$ and may exhibit 
critical points for $E > F(x_c)$. 
This behavior can be observed numerically in the case $a = b = 3$ and $c = 8$.

Let us consider the case $a = b = 2$, which falls under item~(ii) of Case~G. 
We have $x_m \approx -0.0748706245$, $E_0 = F(x_m) \approx 0.0025387196$, 
and $c_0 \approx 2.8004647$. 
Therefore, the period function $T(E)$ is \emph{decreasing} on $(0,\, E_0)$ for all $c \ge c_0$. 
It is not excluded that $T(E)$ remains decreasing for $E > E_0$, 
as can be observed in the case $a = b = 2$ and $c = 3$. 
For $0 < c < c_0$, $T(E)$ may be \emph{increasing} for sufficiently small values of $E$, 
which occurs, for instance, when $a = b = 2$ and $c = 1$.

It is important to note that Theorem~\ref{th_2} does not exhaust all possible monotonicity behaviors of the period function \(T(E)\) for the system under consideration. 
For instance, the case \(c < 0\) has not been discussed here; however, it can be analyzed without significant additional difficulty.

 In general, to investigate a specific case, it suffices to analyze the sign of \(N(x,y)\) within the region
\[
\Omega(E_0) = \{(x,y) \in \mathbb{R}^2 : 0 < H(x,y) \le E_0\}.
\]
Therefore, one needs to determine a critical value \(E_0\)
for which \(\Omega(E_0)\) is the largest domain where \(M(x,y)\) maintains a constant sign.
Specifically, \(E_0\) is chosen so that the energy level \(H(x,y) = E_0\) is tangent to the curve \(M(x,y) = 0\).

\end{remark}

We now present an example of a global unbounded center for which \(M(x,y)\) retains the same sign throughout the entire plane.

\begin{example}[Relativistic harmonic oscillator]
Consider the relativistic harmonic oscillator, whose Hamiltonian is given by
\[
K(q,p) = \gamma m c^2 + \frac{1}{2} k q^2,
\qquad 
\text{where} \quad 
\gamma = \sqrt{1 + \frac{p^2}{m^2 c^2}}.
\]
The first term represents the relativistic kinetic energy, while the second one 
corresponds to the potential energy of a linear restoring force.

Introducing the dimensionless variables
\[
x = \sqrt{\frac{k}{m c^2}}\,q, 
\qquad
y = \frac{p}{m c},
\qquad
H = \frac{K}{m c^2},
\]
the Hamiltonian takes the normalized form
\[
H(x,y) = \frac{1}{2}x^2 + \sqrt{1+y^2} - 1.
\]
Applying Hamilton’s equations, we obtain the system
\[
x' = \frac{y}{\sqrt{1+y^2}}, 
\qquad 
y' = -x,
\]
which corresponds to the classical form of the relativistic harmonic oscillator
(see \cite{babusci2013relativistic}).

In that work, the authors showed, by directly analyzing the expression of the period function \(T(E)\), 
that the oscillation period increases with the system’s energy, as expected from relativistic effects. 
In what follows, we rederive this result as a direct consequence of Theorem~\ref{th1}.

\begin{proposition}
Let \(T(E)\) denote the period function associated with the center at \((0,0)\).
Then \(T(E)\) is an increasing function on \((0,+\infty)\).
\end{proposition}
\end{example}

\begin{proof}
A straightforward computation gives
\[
M(x,y) = \frac{x^4}{2}\,\big((G')^2 - 2G G''\big).
\]
Thus, it suffices to verify that \((G')^2 - 2G G'' \ge 0\) for all \(y\).

Using \(G(y) = \sqrt{1+y^2} - 1\) and setting \(s = \sqrt{1+y^2} \ge 1\), we obtain
\[
(G')^2 - 2G G'' = \frac{(s - 1)^2 (s + 2)}{s^3}.
\]
Since \((s - 1)^2 (s + 2) \ge 0\) for all \(s \ge 1\), with equality only at \(s = 1\)
(i.e., \(y = 0\)), it follows that \((G')^2 - 2G G'' \ge 0\) for all \(y \in \mathbb{R}\).
Hence, \(M(x,y) \ge 0\) for all \((x,y) \in \mathbb{R}^2\).

By Theorem~\ref{th1}, the period function \(T(E)\) is therefore increasing on \((0,+\infty)\).
\end{proof}

 The following two examples were suggested by Chicone in a private
communication, in which he outlined a simple ODE model for the motion of a
liquid slug in an oscillating heat pipe (OHP), as well as a natural
generalization of this model.

\begin{example}[Oscillating heat pipe model]

Consider a straight tube of total length $L>0$, hermetically sealed at both ends
and placed in a horizontal position. The tube is partially filled with a liquid slug
of length $\gamma>0$, which separates two vapor bubbles (plugs) located on the left
and on the right of the slug. Let $u(t)\in(0,L-\gamma)$ denote the position of the 
left endpoint of the slug. Then the lengths of the left and right vapor bubbles are
\[
\ell_L(t)=u(t), \qquad 
\ell_R(t)=L-(u(t)+\gamma).
\]

Each bubble is assumed to behave as an ideal gas with constant amounts of substance
$n_L,n_R>0$ and uniform temperature $T$. Hence the pressures inside the bubbles are
\[
P_L(t)=\frac{n_LRT}{u(t)}, 
\qquad 
P_R(t)=\frac{n_RRT}{L-(u(t)+\gamma)}.
\]
The liquid slug, of mass $m>0$, moves according to Newton's law
\[
m\ddot u(t)=P_L(t)-P_R(t).
\]

Introducing the positive parameters
\[
a=\frac{n_LRT}{m}, 
\qquad 
b=\frac{n_RRT}{m},
\]
the equation of motion becomes
\[
\ddot u=\frac{a}{u}-\frac{b}{L-(u+\gamma)}, 
\qquad 0<u<L-\gamma.
\]

This equation is equivalent to the Hamiltonian system
\[
u'=v, \qquad v'=-V'(u),
\]
with Hamiltonian function
\[
H(u,v)=\tfrac12 v^{2} + V(u), \qquad 
V(u) = -a\ln u - b\ln\!\big(L-(u+\gamma)\big).
\]
The potential $V$ diverges at the endpoints of the domain and has a unique
nondegenerate minimum at
\[
u_0=\frac{a\,(L-\gamma)}{a+b}.
\]

Introducing the translated coordinates
\[
x = u - u_0, \qquad y = v,
\]
the system becomes the Hamiltonian system
\[
x' = y, \qquad y' = -F'(x),
\]
where
\[
F(x)=V(x+u_0)-V(u_0)=a\ln\Big(\frac{u_0}{u_0+x}\Big)+b\ln\Big(\frac{L-\gamma-u_0}{L-\gamma-u_0-x}\Big)
\]
defined on the interval
\[
x \in \bigl(-u_0,\; L-\gamma-u_0\bigr),
\]
and the center is located at $(0,0)$.

\begin{proposition}
Let \(T(E)\) denote the period function associated with the center at \((0,0)\).
Then there exists \(x_0 \in (-u_0,\,L-\gamma-u_0)\) such that \(T(E)\) is strictly
decreasing on the interval \((0,F(x_0))\).
\end{proposition}

\begin{proof}
Since \(G(y)=\tfrac12 y^2\) and \(G'(y)=y\), we have  
\(M(x,y)=\tfrac12\,y^4\,N(x)\), where
\[
N(x)
=6F(x)\big(F''(x)\big)^2
 - 3\big(F'(x)\big)^2 F''(x)
 - 2F(x)F'(x)F'''(x).
\]
Thus the sign of \(M\) is entirely determined by the sign of \(N(x)\).
A direct computation yields
\[
\begin{aligned}
N(x)
&= 6\Bigg(
a\ln\Big(\frac{u_0}{u_0+x}\Big)
 + b\ln\Big(\frac{k}{k-x}\Big)
\Bigg)
\Bigg(
\frac{a}{(x+u_{0})^{2}}
+ \frac{b}{(k-x)^{2}}
\Bigg)^{2} \\[0.4em]
&\quad
- 3\Bigg(
\frac{b}{k-x}-\frac{a}{x+u_{0}}
\Bigg)^{2}
\Bigg(
\frac{a}{(x+u_{0})^{2}}
+ \frac{b}{(k-x)^{2}}
\Bigg) \\[0.4em]
&\quad
- 2\Bigg(
a\ln\Big(\frac{u_0}{u_0+x}\Big)
 + b\ln\Big(\frac{k}{k-x}\Big)
\Bigg)
\Bigg(
\frac{b}{k-x}-\frac{a}{x+u_{0}}
\Bigg)
\Bigg(
\frac{2b}{(k-x)^{3}}
 - \frac{2a}{(x+u_{0})^{3}}
\Bigg),
\end{aligned}
\]
where \(k=L-\gamma-u_0\).

From the explicit expression of \(N\) one shows that there exists \(r>0\) such that
\(N(x)<0\) for all \(x\in(-r,r)\setminus\{0\}\). Moreover,
\[
\lim_{x\to -u_0^+} N(x)=+\infty,
\qquad
\lim_{x\to (L-\gamma-u_0)^-} N(x)=+\infty,
\]
so that \(N(x)>0\) in a neighborhood of each endpoint of the interval
\((-u_0,\,L-\gamma-u_0)\).

By continuity of \(N\) there exist two points
\[
x_-<0<x_+,
\qquad x_\pm\in(-u_0,\,L-\gamma-u_0),
\]
such that \(N(x_\pm)=0\) and
\[
N(x)<0
\quad\text{for all }x\in(x_-,0)\cup(0,x_+).
\]
Let \(x_0\in\{x_-,x_+\}\) be chosen so that
\[
|x_0|=\min\{|x_-|,x_+\}.
\]
Then \(N(x)\le 0\) for all \(x\in[-|x_0|,|x_0|]\).

Therefore, by Theorem~\ref{th1}, the period function \(T(E)\) is strictly 
decreasing on the interval \((0,F(x_0))\).

\end{proof}

Although the result above does not allow us to conclude that \(T(E)\) is decreasing on the whole interval \((0,\infty)\), the physical nature of the model provides no indication that a change of monotonicity should occur.
\end{example}

\begin{example}[Generalized oscillating heat pipe model] Consider the generalized family
\[
\ddot u = \frac{a}{u^{m}} - \frac{b}{(L-(u+\gamma))^{n}}, 
\qquad 0<u<L-\gamma,
\]
where $a,b>0$, $L>0$, $\gamma>0$, and $m,n>0$ are fixed parameters. As before,
$u(t)$ denotes the position of the left endpoint of the liquid slug, so that the
lengths of the left and right vapor plugs are $u(t)$ and $L-(u(t)+\gamma)$,
respectively. The exponents $m$ and $n$ allow for power-law dependences of the
forces generated by the plugs.

This equation can be written in Hamiltonian form
\[
u' = v,\qquad v' = -V'(u),
\]
where
\[
V(u) = -\frac{a}{1-m}\,u^{1-m}
       - \frac{b}{1-n}\,(L-\gamma-u)^{1-n},\quad\text{if }\qquad m\neq 1 \quad\text{and}\quad n\neq 1;
\]

\[
V(u) = -a\ln u - \frac{b}{1-n}\,(L-\gamma-u)^{1-n}, \quad\text{if }\qquad m=1 \quad\text{and}\quad n\neq 1;
\]
\[
V(u) = -\frac{a}{1-m}\,u^{1-m} - b\ln(L-\gamma-u), \quad\text{if }\qquad m\neq 1 \quad\text{and}\quad n=1.
\]
If $m=n=1$ we recover the logarithmic potential considered in the previous example.

In all cases with $m,n>0$ we have
\[
V''(u) = a m\,u^{-m-1} + b n\,(L-\gamma-u)^{-n-1} > 0
\qquad\text{for }0<u<L-\gamma,
\]
so that $V$ is strictly convex on $(0,L-\gamma)$. Moreover,
\[
\lim_{u\to 0^+} V'(u) = -\infty,
\qquad
\lim_{u\to (L-\gamma)^-} V'(u) = +\infty,
\]
and there exists a unique $u_0\in(0,L-\gamma)$ such that $V'(u_0)=0$. Since
$V''(u_0)>0$, this point is a nondegenerate minimum of $V$ and the Hamiltonian
system has a unique center at $(u_0,0)$.

Introducing the translated coordinates
\[
x = u - u_0,\qquad y = v,\qquad k = L-\gamma-u_0,
\]
we obtain a Hamiltonian system of the form
\[
x' = y,\qquad y' = -F'(x),
\]
with
\[
F(x) = V(u_0+x) - V(u_0),
\quad x\in(-u_0,k),
\]
and center at $(0,0)$. 
\end{example}

As in the previous example, the function
\[
N(x)=6F(x)\big(F''(x)\big)^2
      -3\big(F'(x)\big)^2F''(x)
      -2F(x)F'(x)F'''(x)
\]
satisfies \(N(0)=0\), and there exists \(r>0\) such that  
\(N(x)<0\) for all \(x\in(-r,r)\setminus\{0\}\). Moreover,
\[
\lim_{x\to -u_0^+} N(x)=+\infty,
\qquad
\lim_{x\to (L-\gamma-u_0)^-} N(x)=+\infty,
\]
so that \(N(x)>0\) in a neighborhood of each endpoint of the interval
\((-u_0,\,L-\gamma-u_0)\). By continuity of \(N\), there exist two points
\(x_-<0<x_+,\quad x_\pm\in(-u_0,\,L-\gamma-u_0),\)
such that \(N(x_\pm)=0\), and \(N(x)<0\) for all \(x\in(x_-,0)\cup(0,x_+).\) Let \(x_0\in\{x_-,x_+\}\) be chosen so that \(
|x_0|=\min\{|x_-|,\,x_+\}.\) Then \(N(x)\le 0\) for all \(x\in[-|x_0|,\,|x_0|]\). Therefore, by Theorem~\ref{th1}, the period function \(T(E)\) is strictly
decreasing on the interval \((0,F(x_0))\). Hence the generalized model with exponents \(m,n>0\) exhibits the same
monotonicity properties as the original logarithmic case: the period function
does not undergo any change of monotonicity and remains strictly decreasing on
\((0,F(x_0))\).

We conclude this section by applying our theorem to two non-polynomial systems.

\begin{example}[A pair of pendulum--type equations coupled through their velocities]
Consider the planar system
\[x'=\sin y,\quad
y'=-\sin x,
\]
which may be interpreted as a symmetric variant of the harmonic pendulum.

The system is Hamiltonian with the separable Hamiltonian
\[
H(x,y)=F(x)+G(y),\qquad 
F(x)=1-\cos x,\qquad G(y)=1-\cos y.
\]

The origin is a nondegenerate center, and its period annulus is bounded by the
energy level \(E_0=2\), which corresponds to the separatrix connecting the
saddle points \((0,\pm\pi)\) and \((\pm\pi,0)\).

In what follows we analyze the monotonicity of the period function \(T(E)\)
associated with the center at the origin for energies \(E\in(0,2)\).

A direct (though lengthy) computation shows that \(M(x,y)\) can be written in
the factored form
\[
M(x,y)
= 4\,(1-\cos y)\,\sin^{4}\!\Bigl(\frac{x}{2}\Bigr)\,\sin^{2}\!\Bigl(\frac{y}{2}\Bigr)\,N(x,y),
\]
where
\[
N(x,y) = 5 + \cos(2x) - 2(\cos x-2)\cos y.
\]
Hence, the sign of \(M(x,y)\) in any region is entirely determined by the sign
of \(N(x,y)\) in that region. It is not difficult to verify that \(N(x,y)\ge 0\) on
\[
\Omega(2) = \{(x,y)\in\mathbb{R}^2 : 0 \le H(x,y) \le 2\},
\]
and consequently \(M(x,y)\ge 0\) on \(\Omega(2)\). Therefore, by
Theorem~\ref{th1}, the period function \(T(E)\) is monotonically increasing on
the interval \((0,2)\).
\end{example}

\begin{example}
Consider the system
\[x' = \sinh(y),\qquad y' = -\sinh(x).\]

This system is Hamiltonian with
\[
H(x,y)=\cosh(x)-1+\cosh(y)-1,
\]
and it possesses an unbounded global center at the origin.  
In what follows we show that the period function \(T(E)\) associated with this
center is decreasing on the interval \((0,\,4+4\sqrt{3}]\).

A direct calculation yields the factorization
\[
M(x,y)
= -(\cosh x - 1)^2 (\cosh y - 1)^2
\bigl(\cosh^{2}x - \cosh x\,\cosh y + 2\cosh y + 2\bigr).
\]
Defining \(a=\cosh x \ge 1\) and \(b=\cosh y \ge 1\), we obtain
\[
M(x,y)
= -(a-1)^2(b-1)^2 f(a,b),
\qquad
f(a,b)=a^{2}-ab+2b+2.
\]
Consequently, \(M\le 0\) if and only if \(f\ge 0\). Furthermore,
\[
H(x,y)=(a-1)+(b-1)=a+b-2.
\]
Thus, determining the largest \(E_0\) for which \(M\le 0\) on \(\Omega(E_0)\)
reduces to minimizing the sum \(a+b\) along the curve \(f(a,b)=0\). Solving \(f(a,b)=0\) for \(b\) gives
\[
b(a)=\frac{a^{2}+2}{a-2}, \qquad a>2.
\]
A direct computation shows that the function \(S(a)=a+b(a)\) attains its minimum at
\(
a_{\min}=2+\sqrt{3}.
\)
Consequently,
\[
E_0=S(a_{\min})-2=4+4\sqrt{3}.
\]

Therefore, \(M(x,y)\le 0\) for all \((x,y)\) satisfying \(0<H(x,y)\le E_0\), and by
Theorem~\ref{th1} the period function \(T(E)\) is strictly decreasing on the
interval \((0,E_0)\).

The theorem guarantees the monotonicity of \(T(E)\) up to the energy level
\(E_0\). For energies beyond \(E_0\), the derivative \(T'(E)\) may or may not
vanish, and a change of monotonicity cannot be ruled out.

\end{example}
\section{Conclusion}

During the development of this work, two papers were found that address the
monotonicity of the period function for more general planar systems by means of
normalizers (see~\cite{freire2004first, Sabatini2006}). In both references, the
authors apply their methods to discuss the monotonicity of the period function
for the system
\[
x' = G(y), \qquad y' = -F(x).
\]
Furthermore, in the recent paper \cite{villadelprat2020period},
Proposition~3.1 provides a one-dimensional monotonicity condition for separable
Hamiltonians, which is closely related to our Theorem~\ref{th1}.

Nevertheless, the present study remains relevant for three main reasons. First,
it provides a direct generalization of Chicone’s criterion, the most frequently
cited result on this topic. Second, both the analytical expression of \(M(x,y)\)
and its proof are original contributions of this work. Third, the examples
discussed here are quite different from those typically found in the
literature. Moreover, the proposed criterion is remarkably simple to apply,
making it a practical tool for investigating the monotonicity of the period
function in concrete Hamiltonian systems.

We conclude this section with two questions:

\begin{enumerate}

\item What is the explicit expression of the function \(M(x,y)\) in terms of
\(F(x)\) and \(G(y)\) when \(G\) is not even?

\item Under what conditions can a system of the form
\[
u' = P(u,v), \qquad v' = Q(u,v),
\]
with a nondegenerate center at \((0,0)\), be transformed into a Hamiltonian
system of the form
\[
x' = G'(y), \qquad y' = -F'(x),
\]
with \(G\) an even function?

\end{enumerate}

In \cite{nascimento2024analytic}, it was proven that if \(P\) and \(Q\) are
analytic functions in a neighborhood of \((0,0)\), and if \((0,0)\) is a
nondegenerate center, then there exists a change of variables that transforms
the original system into a Hamiltonian system of potential type,
\[
x' = y, \qquad y' = -F'(x).
\]
In other words, the problem can be solved locally.

In \cite{grotta2025global}, the authors considered a simplified version of this
problem in a global context. More precisely, they studied systems of the form
\[
P(u,v) = v, \qquad Q(u,v) = f(u,v^2),
\]
where \(f\) is analytic on \(\mathbb{R}^2\), and showed that if \((0,0)\) is a
nondegenerate center (and the only equilibrium), then there exists an analytic
change of variables that transforms the system into the Hamiltonian form
\[
x' = y, \qquad y' = -F'(x).
\]

\section*{Acknowledgements}

The author expresses his sincere gratitude to Professor Clodoaldo G. Ragazzo for his valuable teachings (to whom this work is dedicated on the occasion of his 60th birthday), and to Professor C. Chicone for suggesting the examples related to the oscillating heat pipe (OHP) model.

\bibliography{refs} 
\end{document}